\documentclass[12pt]{amsart}
\usepackage[utf8]{inputenc}
\usepackage{mathrsfs} 
\usepackage{amssymb}
\usepackage{bm}
\usepackage{graphicx}
\usepackage[centertags]{amsmath}
\usepackage{amsfonts}
\usepackage{amsthm}
\usepackage{enumerate}
\usepackage{tocvsec2}
\usepackage{xcolor}
\usepackage[margin=1in]{geometry}

\newtheorem{theorem}{Theorem}[section]
\newtheorem*{theorem*}{Theorem}

\newtheorem{corollary}[theorem]{Corollary}
\newtheorem{lemma}[theorem]{Lemma}
\newtheorem{rem}[theorem]{Remark}

\newtheorem{proposition}[theorem]{Proposition}

\theoremstyle{definition}

\newcommand{\ee}{\varepsilon}
\newcommand{\nn}{\mathbb{N}}

\begin{document}
\title[Asymptotically uniformly flat norms]
{Power type asymptotically uniformly smooth \\ and asymptotically uniformly flat norms}
\author{R.M. Causey}

\begin{abstract} We provide a short characterization of $p$-asymptotic uniform smoothability and asymptotic uniform flatenability of operators and of Banach spaces. We use these characterizations to show that many asymptotic uniform smoothness properties pass to injective tensor products of operators and of Banach spaces. In particular, we prove that the injective tensor product of two asymptotically uniformly smooth Banach spaces is asymptotically uniformly smooth.  We prove that for $1<p<\infty$, the class of $p$-asymptotically uniformly smoothable operators can be endowed with an ideal norm making this class a Banach ideal. We also prove that the class of asymptotically uniformly flattenable operators can be endowed with an ideal norm making this class a Banach ideal.

\end{abstract}

\maketitle

\tableofcontents

\addtocontents{toc}{\setcounter{tocdepth}{1}}

\section{Introduction}

Beginning with Kloeckner's elegant proof \cite{Kl} of Bourgain's theorem \cite{Bo} that binary trees of arbitrary, finite height cannot be uniformly biLipschitzly embedded into uniformly convex Banach spaces, a number of self-improvement arguments have appeared in the literature to prove results in the non-linear  theory of Banach spaces and of operators (see, for example, \cite{B}, \cite{BZ}, \cite{CD1}, \cite{CD}, \cite{JS}).  These self-improvement arguments each involve an isometric property of a Banach space or of an operator. That is, each of these arguments depends upon some property which may be gained or lost by equivalently renorming the space, or the domain or range of the operator. Consequently, renorming theorems have been of significant use in the recent advances of the non-linear theory. The notion of a $p$-asymptotically uniformly smooth Banach space has recently seen significant use in the non-linear theory of Banach spaces. Indeed, the existence of asymptotically uniformly smooth norms,  $p$-asymptotically smooth norms, or asymptotically uniformly flat norms on Banach spaces have been used in \cite{BZ}, \cite{CD}, \cite{DKLR},  \cite{GKL1}, \cite{GKL2}, \cite{KR},  and \cite{LR}.  For some of these results, such as \cite{KR} and \cite{LR}, the precise results depend upon knowing for which $p$ an asymptotically uniformly smoothable Banach space admits an $p$-asymptotically  smooth equivalent norm.   In \cite{CD}, the notion of an asymptotically uniformly smooth operator was used to extend non-linear results from \cite{BKL} previously known only in the spatial case.  Therefore it is of interest to offer isomorphic characterizations of when a Banach space admits an equivalent norm which is asymptotically $p$-smooth. The first part of this work is devoted to defining isomorphic properties which characterize those $p$ such that a given operator or Banach space admits an equivalent $p$-asymptotically smooth norm. To that end, we have the following result. All necessary terminology will be given in the second section.

\begin{theorem} Let $A:X\to Y$ be an operator. \begin{enumerate}[(i)]\item For any $1<p<\infty$, $A$ is $p$-asymptotically uniformly smoothable if and only if it satisfies $\ell_p$ upper tree estimates.  \item $A$ is asymptotically uniformly flattenable if and only if it satisfies $c_0$ upper tree estimates. \end{enumerate}

\label{main1}
\end{theorem}

Applying Theorem \ref{main1} to the identity operator of a Banach space yields new renorming theorems for non-separable Banach spaces. In the case that $A=I_X$ and $X$ is separable, Theorem \ref{main1} follows from \cite{FOSZ}. However, that proof uses the separability in a fundamental way and cannot be modified to work for the non-separable case. Furthermore, \cite{FOSZ} deduces the content of Theorem \ref{main1} for separable spaces from a separate, quite involved result. Our proof  is short and direct, and is in the spirit of Pisier's famous renorming theorem \cite{P}. Furthermore, our proof holds for operators and does not use separability in any way. 

The reader should compare Theorem \ref{main1} to the beautiful renorming theorem from \cite{GKL1}, and the generalization in \cite{C} to operators and non-separable spaces. 

\begin{theorem}\cite{GKL1} If $X$ is a separable, infinite dimensional Banach space, then the Szlenk power type $\textbf{\emph{p}}(X)$ lies in $[1, \infty)$, and if $1/q+1/\textbf{\emph{p}}(X)=1$, $q$ is the supremum of those $p$ such that $X$ admits an equivalent, $p$-asymptotically uniformly smooth norm. 

\label{gkl}
\end{theorem}

We invite the comparison because Theorem \ref{gkl} provides an isomorphic characterization of the supremum of those $p$ for which $X$ admits an equivalent $p$-asymptotically uniformly smooth norm, but it does not answer the question of whether this supremum is attained. Indeed, for any $1<p\leqslant \infty$ and $q$ such that $1/p+1/q=1$, one can exhibit two spaces $X_p$, $Y_p$ such that $\textbf{p}(X_p)=\textbf{p}(Y_p)=q$ and such that $X_p$ is $p$-asymptotically uniformly smooth (resp. asymptotically uniformly flat), while $Y_p$ fails to be $p$-asymptotically uniformly smoothable (resp. asymptotically uniformly flattenable). Indeed, one can take $X_p=\ell_p$ (resp. $c_0$ if $p=\infty$), and let $Y_p$ be the dual of the $q$-convexification $T^{(q)}$ of the Figiel-Johnson Tsirelson space. Theorem \ref{main1} solves this problem of isomorphically determining which spaces (and operators) attain the supremum mentioned in Theorem \ref{gkl}.

Along the way, we offer a short proof of the fact that asymptotic moduli pass in the nicest way possible to injective tensor products of operators. This extends a recent result of \cite{GLR}, in which it was shown that the property of strong asymptotic uniform smoothness passes of two Banach spaces implies asymptotic uniform smoothness of the injective tensor product.  This result also has as a corollary one of the main theorems from \cite{DK}, in which it was shown that the Szlenk power type of an injective tensor product of two non-zero Banach spaces is the maximum of the Szlenk power type of the individual spaces.

\begin{theorem} Let $A_0:X_0\to Y_0$, $A_1:X_1\to Y_1$ be operators and fix $1<p<\infty$.  If $A_0$, $A_1$ each have any of the following properties, so does the injective tensor product $A_0\otimes A_1:X_0\hat{\otimes}_\ee X_1\to Y_0\hat{\otimes}_\ee Y_1$: \begin{enumerate}[(i)]\item Asymptotic uniform smoothness. \item $p$-asymptotic uniform smoothness. \item Asymptotic uniform flatness. \end{enumerate}

\label{main2}

\end{theorem}

Of course, this theorem concerning injective tensor products gives information on certain spaces of compact operators. 

\begin{theorem} Let $X,Y$ be Banach spaces such that either $X^*$ or $Y$ has the approximation property. Then for any $1<p<\infty$, the space $\mathcal{K}(X,Y)$ of compact operators from $X$ into $Y$ has the following properties as long as $X^*$ and $Y$ do. \begin{enumerate}[(i)]\item Asymptotic uniform smoothness. \item $p$-asymptotic uniform smoothness. \item Asymptotic uniform flatness. \end{enumerate}

\noindent In particular, if $X$ is reflexive, $q$-asymptotically uniformly convex, and $Y$ is $p$-asymptotically uniformly smooth, where $1/p+1/q=1$, then $\mathcal{K}(X,Y)$ is $p$-asymptotically uniformly smooth.

\end{theorem}

In the final section of the paper, we study classes of $p$-asymptotically uniformly smoothable operators, $\mathfrak{T}_p$, and asymptotically uniform flattenable operators, $\mathfrak{T}_\infty$. Regarding these classes, we prove the following.

\begin{theorem} For each $1<p\leqslant \infty$, there exists an ideal norm $\mathfrak{t}_p$ on $\mathfrak{T}_p$ making $(\mathfrak{T}_p, \mathfrak{t}_p)$ a Banach ideal.

\end{theorem}

\section{Definitions and main theorems}

Throughout, we will work over the scalar field $\mathbb{K}$, which is either $\mathbb{R}$ or $\mathbb{C}$. By ``operator,'' we shall mean continuous, linear operator. 

Given an operator $A:X\to Y$, we define the \emph{modulus of asymptotic uniform smoothness} of a non-zero operator $A$ by $$\rho(\sigma)=\sup\Bigl\{\underset{\lambda}{\lim\sup} \|y+Ax_\lambda\|-1: y\in B_Y, (x_\lambda)\subset \sigma B_Y, x_\lambda\underset{\text{weak}}{\to} 0\Bigr\}$$ for $\sigma>0$  and $\rho(0, A)=0$, and the \emph{modulus of weak}$^*$ \emph{asymptotic uniform convexity} of $A$ by $$\delta^{\text{weak}^*}(\tau) = \inf\bigl\{\|y^*+y^*_\lambda\|-1: \|y^*\|=1, \|A^*y^*_\lambda\|\geqslant \tau, y^*_\lambda\underset{\text{weak}^*}{\to}0\bigr\}$$  for $\tau>0$.  We define $\rho(\sigma, A)=0$  when $A$ is the zero operator.  We remark that $\rho(\sigma, A)=0$ for every $\sigma>0$ if and only if $A$ is compact, and $\delta^{\text{weak}^*}(\tau, A)=\infty$ for every $\tau>0$ if and only if $A$ is compact. Otherwise $\delta^{\text{weak}^*}(\tau, A)<\infty$ for every $\tau>0$, and $\delta^{\text{weak}^*}(\cdot, A)$ is continuous. 

We say $A:X\to Y$ is \emph{asymptotically uniformly smooth} if $\lim_{\sigma\to 0^+}\rho(\sigma, A)/\sigma=0$. It is easy to see that $\rho(\cdot, A)$ is convex, so that $\lim_{\sigma\to 0^+}\rho(\sigma, A)/\sigma=\inf_{\sigma>0}\rho(\sigma, A)/\sigma$. For $1<p<\infty$, we say $A:X\to Y$ is $p$-\emph{asymptotically uniformly smooth} provided $\sup_{\sigma>0}\rho(\sigma, A)/\sigma^p<\infty$.  It is easy to see that this is equivalent to the property that there exists $\sigma_1>0$ such that $\sup_{\sigma<\sigma_1} \rho(\sigma, A)/\sigma^p<\infty$.  Finally, we introduce a property previously defined for Banach spaces, but which has not previously been introduced in the literature for operators. We say $A:X\to Y$ is \emph{asymptotically uniformly flat} provided there exists $\sigma>0$ such that $\rho(\sigma, A)=0$.  

We remark on the fact that our definition of the $\rho(\cdot, A)$ modulus, we take the supremum over $y\in B_Y$, while the definition sometimes only takes the supremum over $y\in S_Y$. One can see through the usual duality between $\rho(\cdot, A)$ and $\delta^{\text{weak}^*}(\cdot, A)$, which can be proved using either definition of $\rho(\cdot, A)$, that these differences do not affect whether or not $A$ is asymptotically uniformly smooth, $p$-asymptotically uniformly smooth, or asymptotically uniformly flat.

We note that if $A:X\to Y$ is asymptotically uniformly smooth (resp. $p$-asymptotically uniformly smooth, asymptotically uniformly flat), and if $|\cdot|$ is an equivalent norm on $X$, then $A:(X, |\cdot|)\to Y$ is also asymptotically uniformly smooth (resp. $p$-asymptotically uniformly smooth, asymptotically uniformly flat). Therefore when looking for equivalent norms which gain one of these properties, we are interested only in renorming $Y$. To that end, we say $A:X\to Y$ is \emph{asymptotically uniformly smoothable} if there exists an equivalent norm $|\cdot|$ on $Y$ such that $A:X\to (Y, |\cdot|)$ is asymptotically uniformly smooth. The notions of $p$-asymptotically uniformly smoothable and asymptotically uniformly flattenable are defined similarly. 

A Banach space is said to be asymptotically uniformly smooth, asymptotically uniformly smoothable, $p$-asymptotically uniformly smooth, etc., if its identity operator has that property. We remark that if $I_X:X\to X$ is asymptotically uniformly smoothable (resp. $p$-asymptotically uniformly smoothable, asymptotically uniformly flattenable), there exists an equivalent norm $|\cdot|$ on $X$ such that $I_X:X\to (X, |\cdot|)$ has the corresponding isometric property. By our remarks above, $I_X:(X, |\cdot|)\to (X, |\cdot|)$ also has the same isometric property. Therefore while our renorming theorems only explicitly refer to renorming the range space of the operator, the spatial results we obtain need not distinguish between the space $X$ as the domain or range of its identity.

We define other important quantities related to the $\rho(\cdot, A)$ modulus. For $1<p<\infty$, we let $\textbf{t}_p(A)$ be the infimum of those $C>0$ such that for any $y\in Y$, any $\sigma>0$, and any weakly null net $(x_\lambda)\subset \sigma B_X$, $$\underset{\lambda}{\lim\sup} \|y+Ax_\lambda\|^p \leqslant \|y\|^p + C^p\sigma^p.$$  We obey the convention that the infimum of the empty set is $\infty$.   We let $\textbf{t}_\infty(A)$ be the infimum of those $C>0$ such that for any $y\in Y$, any $\sigma>0$, and any weakly null net $(x_\lambda)\subset \sigma B_X$, $$\underset{\lambda}{\lim\sup} \|y+Ax_\lambda\|\leqslant \max\{\|y\|, C\sigma \}.$$  These quantities are related to the $\rho(\cdot, A)$ modulus by the following proposition. The proof of the proposition involves only some basic calculus, so we omit it. 

\begin{proposition} Let $A:X\to Y$ be an operator. \begin{enumerate}[(i)]\item For any $1<p<\infty$, there exist functions $h_p:[0, \infty)^2\to [0, \infty)$, $f_p:[0, \infty)\to [0, \infty)$ such that if $\textbf{\emph{t}}_p(A)\leqslant C$, $\sup_{\sigma>0}\rho(\sigma, A)/\sigma^p\leqslant f_p(C)$, and if $\sup_{\sigma>0} \rho(\sigma, A)/\sigma^p \leqslant C$ and $\|A\|\leqslant C'$, then $\textbf{\emph{t}}_p(A)\leqslant h_p(C, C')$. \item If $\sigma_0(A)=\inf_{\sigma>0} \rho(1/\sigma, A)=0$, then $\sigma_0(A)\leqslant \textbf{\emph{t}}_\infty(A)\leqslant \sigma_0(A)+\|A\|$.  \end{enumerate}

In particular, $A$ is $p$-asymptotically uniformly smooth if and only if $\textbf{\emph{t}}_p(A)<\infty$, and $A$ is asymptotically uniformly flat if and only if $\textbf{\emph{t}}_\infty(A)<\infty$.

\label{related moduli}

\end{proposition}

Given a directed set $D$, we say a subset $D_0$ of $D$ is \begin{enumerate}[(i)]\item \emph{cofinal} if for any $u\in D$, there exists $v\in D_0$ such that $u\leqslant v$,  \item \emph{eventual} if its complement is not cofinal. \end{enumerate}

We remark that a finite intersection of eventual sets is eventual.

Given a directed set $D$, we let $D^n=\{(u_i)_{i=1}^n: u_i\in D\}$, and $D^{<\nn}=\cup_{n=1}^\infty D^n$. We order $D^{<\nn}$ by letting $s<t$ if for some $u_1, \ldots, u_n$ and $1\leqslant m\leqslant n$, $s=(u_i)_{i=1}^m$ and $t=(u_i)_{i=1}^n$. We let $s\smallfrown t$ denote the concatenation of $s$ and $t$. Given $t=(u_i)_{i=1}^n$ and $1\leqslant m\leqslant n$, we let $t|_m=(u_i)_{i=1}^m$.  We say a subset $B$ of $D^{<\nn}$ is \begin{enumerate}[(i)]\item \emph{inevitable} if $D^1\cap B$ is eventual and for any $n\in \nn$ and $t\in D^n\cap B$, $\{s\in D^{n+1}: t<s\}\cap B$ is eventual, \item \emph{big} if it contains an inevitable subset. \end{enumerate}

It follows easily from the definitions and from the fact that a finite intersection of eventual sets is eventual that a finite intersection of inevitable sets is inevitable and a finite intersection of big sets is big.

If $X$ is a Banach space and $D$ is a weak neighborhood basis at $0$ in $X$, we say a collection $(x_t)_{t\in D^{<\nn}}\subset X$ is \emph{normally weakly null} provided that for any $t=(u_i)_{i=1}^n\in D^{<\nn}$, $ x_t\in u_n$.

Given a Banach space $E$ with basis $(e_i)_{i=1}^\infty$, an operator $A:X\to Y$, and a weak neighborhood basis $D$ at $0$ in $X$, we say $A:X\to Y$ \emph{satisfies} $C$-$E$ \emph{upper tree estimates} provided that for any normally weakly null $(x_t)_{t\in D^{<\nn}}\subset B_X$, $$\bigcup_{n=1}^\infty \Bigl\{t\in D^n: (\forall (a_i)_{i=1}^n\in \mathbb{K}^n)\Bigl(\|A\sum_{i=1}^n a_i x_{t|_i}\|\leqslant C\|\sum_{i=1}^n a_i e_i\|_E\Bigr)\Bigr\}$$ is big. Here, of course, $D$ is ordered by reverse inclusion.  We note that the definition is made in terms of a weak neighborhood basis $D$. We will show in Corollary \ref{shortcut} that the property of satisfying $C$-$E$ upper tree estimates is independent of the choice of $D$, which justifies our regular practice of leaving $D$ unspecified.  We say $A:X\to Y$ satisfies $E$-\emph{upper tree estimates} if there exists some $C>0$ such that $A$ satisfies $C$-$E$ upper tree estimates.   We will be concerned with the case $E=c_0$ or $E=\ell_p$ for some $1<p<\infty$.  

We note that for Banach spaces with separable dual, the notion of $\ell_p$ or $c_0$ upper tree estimates has already been defined in the literature (see, for example, \cite{FOSZ}). Our presentation of the definition differs from what is commonly given in the literature, but we discuss in Section $3$ how our presentation differs from the usual one, but the underlying property coincides with the usual one.

We make one easy observation before moving to our renorming theorem.

\begin{proposition} Let $A:X\to Y$ be an operator. Suppose that $1<p<\infty$, $G\subset B_Y$, $C, \sigma>0$ are such that $\overline{\text{\emph{co}}}(G)=B_Y$, for any $y\in G$, and for any weakly null net $(x_\lambda)\subset \sigma B_X$, $$\underset{\lambda}{\lim\sup} \|y+Ax_\lambda\| \leqslant 1+C\sigma^p.$$ Then $\rho(\sigma, A)\leqslant C\sigma^p$.   

\label{dpt}

\end{proposition}

\begin{proof}Fix $y\in B_Y$, $\sigma>0$, and a weakly null net $(x_\lambda)\subset \sigma B_X$. Fix $\delta>0$, positive numbers $a_1, \ldots, a_n$ summing to $1$, and $y_1, \ldots, y_n\in G$ with $\|y-\sum_{i=1}^n a_i y_i\|<\delta$.   Then $$\underset{\lambda}{\lim\sup} \|y+Ax_\lambda\| \leqslant \delta +\sum_{i=1}^n a_i \underset{\lambda}{\lim\sup} \|y_i+Ax_\lambda\| \leqslant \delta +1+C\sigma^p.$$ This easily yields the result.

\end{proof}

We now turn to the proof of Theorem \ref{main1}. 

\begin{theorem} Fix an operator $A:X\to Y$. \begin{enumerate}[(i)]\item $A$ is asymptotically uniformly flattenable if and only if $A$ satisfies $C$-$c_0$ upper tree estimates. \item For any $1<p<\infty$, $A$ is $p$-asymptotically uniformly smoothable if and only if  $A$ satisfies $C$-$\ell_p$ upper tree estimates. \end{enumerate}

\label{renorming1}

\end{theorem}

\begin{proof}[Proof of Theorem \ref{renorming1}] $(i)$ Assume $\rho(1/\sigma_0, A)=0$ and note that $\textbf{t}_\infty(A)<\infty$ by Proposition \ref{related moduli}. Fix any number $C>\textbf{t}_\infty(A)$. Fix positive numbers $(\ee_n)_{n=1}^\infty$ such that $\textbf{t}_\infty(A)+\sum_{i=1}^\infty \ee_i <C$.  Fix $(x_t)_{t\in D^{<\nn}}\subset B_X$. Let $x_\varnothing=0$. For each $n\in \nn$, let $$B_n'=\Bigl\{t\in D^n: (\forall (a_i)_{i=0}^n\in B_{\ell_\infty^n} )\Bigl(\|A\sum_{i=0}^n a_ix_{t|_i}\|\leqslant \ee_n + \max\{\|A\sum_{i=0}^{n-1} a_i x_{t_i}\|, \textbf{t}_\infty(A)|a_n|\}\Bigr)\Bigr\}.$$ It follows from compactness of $B_{\ell_\infty^n}$ and the definition of $\textbf{t}_\infty(A)$ that  $$B':=\{t\in D^{<\nn}:(\forall 1\leqslant i\leqslant |t|)(t|_i\in B_i')\}$$ is inevitable. It then follows that $$B=\bigcup_{n=1}^\infty \{t\in B_n: (\forall (a_i)_{i=1}^n\in \mathbb{K}^n)(\|A\sum_{i=1}^n a_i x_{t|_i}\|\leqslant C\|(a_i)_{i=1}^n\|_{\ell_\infty^n}\}\supset B'$$ is big. Thus $A$ satisfies $C$-$c_0$ upper tree estimates for any $C>\textbf{t}_\infty(A)$.  

Now assume $Y$ satisfies $C$-$c_0$ upper tree estimates. Let us choose a weak neighborhood basis $D$ at $0$ in $X$ consisting of balanced sets.  For each $y\in Y$, let $g(y)$ denote the infimum of those $C_1$ such that for any normally weakly null $(x_t)_{t\in D^{<\nn}}\subset B_X$, $$B_y(C_1):=\bigcup_{n=1}^\infty \Bigl\{t\in D^n:\|y+A\sum_{i=1}^n x_{t|_i}\| \leqslant C_1\Bigr\}$$ is big.  We first claim that $g$ satisfies the following properties: 

\begin{enumerate}[(a)]\item $g$ is $1$-Lipschitz, \item $\|y\|\leqslant g(y)\leqslant \|y\|+C$ for any $y\in Y$,  \item for any unimodular scalar $\ee$, $g(\ee y)=g(y)$, \item $g$ is convex. \end{enumerate} Items $(a)$-$(c)$ are evident from the definitions of $C$-$c_0$ upper tree estimates, normally weakly null, and $g$. For convexity, fix $y, y'\in Y$ and $0<\alpha<1$. Fix $C_1>g(y)$ and $C_1'>g(y')$. Fix a normally weakly null $(x_t)_{t\in D^{<\nn}}\subset B_X$ and let $B_{\alpha y+(1-\alpha)y'}(\alpha C_1+(1-\alpha)C_1')$, $B_y(C_1)$, and $B_{y'}(C_1')$ be as in the definition of $g$. Then by the triangle inequality, $$B_{\alpha y+(1-\alpha)y'}(\alpha C_1+(1-\alpha)C_1')\supset B_y(C_1)\cap B_{y'}(C_1').$$ Since the intersection of two big sets is big and a superset of a big set is big, $B_{\alpha y+(1-\alpha)y'}(\alpha C_1+(1-\alpha)C_1')$ is big. Since this holds for any $C_1>g(y)$ and $C_1'>g(y')$ and any normally weakly null $(x_t)_{t\in D^{<\nn}}\subset B_X$,  $g(\alpha y+(1-\alpha)y') \leqslant \alpha g(y)+(1-\alpha) g(y')$, and $g$ is convex.

Now let $G=\{y\in Y: g(y)\leqslant 1+C\}$. Note that by $(a)$-$(d)$ above, this is the unit ball of a norm $|\cdot|$ on $Y$. If $y\in B_Y$, then $g(y)\leqslant \|y\|+C\leqslant 1+C$, and $y\in G$. If $g(y)\leqslant 1+C$, then $\|y\|\leqslant g(y)\leqslant 1+C$, so $B_Y\subset G\subset (1+C)B_Y$, and $|\cdot|$ is $(1+C)$-equivalent to $\|\cdot\|$. We finally claim that $\rho(1, A:X\to (Y, |\cdot|))=0$.  For this, it is sufficient to show that for any weakly null net $(x_\lambda)\subset B_X$ and any $y\in G$, $$\underset{\lambda}{\lim\sup}g(y+Ax_\lambda)\leqslant 1+C.$$  To obtain a contradiction, assume $\mu>0$, $y\in G$, and $(x_u)_{u\in D}$ are such that $x_u\in u$ and $g(y+Ax_u)\geqslant 1+C+2\mu$ for all $u\in D$. Now for each $u\in D$, we may fix a normally weakly null collection $(x^u_t)_{t\in D^{<\nn}}\subset B_X$ such that with $$B_u=\bigcup_{n=1}^\infty \Bigl\{t\in D^{<\nn}: \|y+Ax_u+A\sum_{i=1}^n x_{t|_i}^u\|\leqslant 1+C+\mu\Bigr\},$$ $B_u$ is not big. We now let $x_{(u)}=x_u$ and $x_{(u)\smallfrown t}=x^u_t$. Now if $$B'\subset B:=\bigcup_{n=1}^\infty \Bigl\{t\in D^n: \|y+A\sum_{i=1}^n x_{t|_i}\|\leqslant 1+C+\mu\Bigr\}$$ is inevitable and $u\in B'\cap D^1$, then $B_u$ would be big. Since no $B_u$ is big, we deduce that $B$ cannot be big, which contradicts $g(y)\leqslant 1+C$.  This contradiction finishes $(i)$.

$(ii)$ We argue in a manner similar to $(i)$. If $\sup_{\sigma>0} \rho(\sigma, A)/\sigma^p<\infty$, then by Proposition \ref{related moduli}, $\textbf{t}_p(A)<\infty$. Now fix $C>\textbf{t}_p(A)$ and a sequence $(\ee_n)_{n=1}^\infty$ of positive numbers such that $\textbf{t}_p(A)^p+\sum_{n=1}^\infty \ee_n <C^p$.  We argue as in $(i)$ to find a big set $B\subset D^{<\nn}$ such that for any $n\in \nn$ and $t\in B\cap D^n$, $$\|A\sum_{i=1}^n a_i x_{t|_i}\|^p \leqslant \|A\sum_{i=1}^{n-1} a_i x_{t|_i}\|^p + \textbf{t}_p(A)^p |a_n|^p +\ee_n$$ for any $(a_i)_{i=1}^\infty \in B_{\ell_p^n}$. Iterating, we deduce that for any $t\in B\cap D^n$ and any $(a_i)_{i=1}^n\in B_{\ell_p^n}$, $$\|A\sum_{i=1}^n a_i x_i\|^p \leqslant \|A\sum_{i=1}^{n-1} a_i x_{t|_i}\|^p+\textbf{t}_p(A)^p |a_n|^p +\ee_n \leqslant\ldots \leqslant \textbf{t}_p(A)^p\sum_{i=1}^n |a_i|^p+ \sum_{i=1}^n \ee_i <C^p.$$  From this we can easily  deduce that $A$ satisfies $C$-$\ell_p$ upper tree estimates.

Now assume that $A$ satisfies $C_1$-$\ell_p$ upper tree estimates with $C_1\geqslant 1$.  Then with $C=2C_1 $, we deduce that for any $y\in Y$ and any normally weakly null $(x_t)_{t\in D^{<\nn}}\subset B_X$, $$B_y:=\bigcup_{n=1}^\infty \Bigl\{t\in D^n: (\forall (a_i)_{i=1}^n \in \ell_p^n)(\|y+A\sum_{i=1}^n a_i x_{t|_i}\|^p\leqslant C^p\bigl(\|y\|^p+ \sum_{i=1}^n |a_i|^p\bigr)\Bigr\}$$ is big.  Indeed, we note that if for some $t\in D^n$, $\|A\sum_{i=1}^n a_i x_{t_i}\|\leqslant C_1 \|(a_i)_{i=1}^n\|_{\ell_p^n}$, then $$\|y+A\sum_{i=1}^n a_i x_{t|_i}\| \leqslant C_1\Bigl(\|y\|+\bigl(\sum_{i=1}^n |a_i|^p\bigr)^{1/p}\Bigr) \leqslant 2C_1  \Bigl(\|y\|^p + \sum_{i=1}^n |a_i|^p\Bigr)^{1/p}.$$   

For $y\in Y$, let $[y]$ be the infimum of those $C_2>0$ such that for any normally weakly null $(x_t)_{t\in D^{<\nn}}\subset B_X$, $$B_y(C_2):=\bigcup_{n=1}^\infty \Bigl\{t\in D^n: (\forall (a_i)_{i=1}^n \in \mathbb{K}^n)\Bigl(\bigl(\frac{1}{C^p}\bigl\|y+A\sum_{i=1}^n a_i x_{t|_i}\bigr\|^p - \sum_{i=1}^n |a_i|^p\bigr)^{1/p} \leqslant C_2\Bigr)\Bigr\}$$ is big.  Note that for any $y\in Y$, $\frac{1}{C}\|y\|\leqslant [y]\leqslant \|y\|$ and for any $y\in Y$ and any scalar $c$, $[cy]=|c|[y]$. Now let $|y|=\inf\Bigl\{\sum_{i=1}^n [y_i]: y=\sum_{i=1}^n y_i\Bigr\}$.  Let $G=\{y\in Y: [y]<1\}$ and note that $\overline{\text{co}}(G)=B_Y^{|\cdot|}$. Moreover, $|\cdot|$ is an equivalent norm on $Y$ such that $\frac{1}{C}\|y\|\leqslant |y|\leqslant \|y\|$ for all $y\in Y$.

Now fix any $y\in G$, $\sigma>0$,  and a weakly null net $(x_\lambda)\subset \sigma B_X$.  Assume that $\mu>0$ is such that $$\underset{\lambda}{\lim\sup} [y+A x_\lambda]^p> 1+ \mu^p.$$  Then there exists $(x_u)_{u\in D}\subset  B_X$ such that $x_u\in u$ for all $u\in D$ and $\inf_{u\in D} [y+A\sigma x_u]^p> 1+\mu^p.$ Then for every $u\in D$, we may fix $(x^u_t)_{t\in D^{<\nn}}\subset B_X$ normally weakly null such that  $$B_u=\bigcup_{n=1}^\infty \Bigl\{t\in D^n: (\forall (a_i)_{i=1}^n \in \mathbb{K}^n)\Bigl(\frac{1}{C^p}\|y+A\sigma x_u+A\sum_{i=1}^n a_i x_{t|_i}^u\|^p-\sum_{i=1}^n |a_i|^p\leqslant 1+\mu^p\Bigr)\Bigr\}$$ fails to be big. Now let $x_{(u)}=x_u$ and $x_{(u)\smallfrown t}=x^u_t$.  Since $[y]<1$,  there exists an inevitable subset $B$ of $D^{<\nn}$ which is contained in $$\bigcup_{n=1}^\infty \Bigl\{t\in D^n: (\forall (a_i)_{i=1}^n\in \mathbb{K}^n)\Bigl(\frac{1}{C^p}\|y+A\sum_{i=1}^n x_{t|_i}\|^p - \sum_{i=1}^n |a_i|^p \leqslant 1\Bigr)\Bigr\}.$$  Now for any $u_1\in B$, $$B'_{u_1}:=\Bigl\{t\in D^n: (\forall (a_i)_{i=1}^n\in \mathbb{K}^{n+1})\Bigl(\frac{1}{C^p}\|y+A\sigma x_{(u_1)}+A\sum_{i=1}^n a_i x_{(u)\smallfrown (t|_i)}\|^p -\sigma^p-\sum_{i=1}^n |a_i|^p\leqslant 1\Bigr\}$$ is big. Since $B_{u_1}$ is not big, there exists $t \in B_{u_1}'\setminus B_{u_1}$.  This means there exist $n\in \nn$ and $(a_i)_{i=1}^n\in \mathbb{K}^n$ such that $$ \frac{1}{C^p}\|y+A\sigma x_{(u_1)}+A\sum_{i=1}^n a_i x_{t|_i}^{u_1}\|^p- \sum_{i=1}^n |a_i|^p \geqslant 1+\mu^p.$$   However, since $t\in B_{u_1}$, $$ \frac{1}{C^p} \|y+A\sigma x_{(u_1)}+A\sum_{i=1}^n a_i x_{t|_i}^{u_1}\|^p - \sum_{i=1}^n |a_i|^p \leqslant 1+\sigma^p.$$  This yields that $\mu\leqslant \sigma$.  From this it follows that $$\underset{\lambda}{\lim\sup} |y+Ax_\lambda|-1\leqslant \underset{\lambda}{\lim\sup} [y+Ax_\lambda] -1\leqslant (1+\sigma^p)^{1/p}-1 \leqslant \sigma^p.$$  Since this holds for any $y\in G$, we deduce that $\rho(\sigma, A)\leqslant \sigma^p$. Since this holds for any $\sigma>0$, $\sup_{\sigma>0} \rho(\sigma, A)/\sigma^p\leqslant 1$.

\end{proof}

\begin{rem}\upshape It follows from the proof of Theorem \ref{renorming1} that for any $1<p\leqslant \infty$, $T_p(A)\leqslant \textbf{t}_p(A)$. 

It also follows that for any $1<p\leqslant \infty$ and Proposition \ref{related moduli}, there exists a constant $C_p\geqslant 1$ such that if $T_p(A)<1$ and $\|A\|<1 $, there exists an equivalent norm $|\cdot|$ on $Y$ such that $\frac{1}{2}B_Y\subset B_Y^{|\cdot|}\subset 2 B_Y$ and such that $\textbf{t}_p(A:X\to (Y, |\cdot|))\leqslant C_p$. 

\label{useful remark}

\end{rem}

\section{An aside on games}

In this section, we discuss the definition of $C$-$\ell_p$ upper tree estimates and its implications. For technical reasons, it is much easier to prove the results of the previous section having the definition of $C$-$\ell_p$ upper tree estimates in terms of big and inevitable sets. Upper tree estimates are usually given in the form of ``every weakly null tree has a branch with a given property,'' while our definition is that for every normally weakly null tree, ``almost all'' branches have this given property, where the notion of ``almost all'' is comes from our definition of ``big.'' However, this presentation perhaps makes the notion less clear. The purpose of this section is to provide some intuition behind the definition by framing it in terms of winning strategies in two player games, and to show that these two presentations coincide.  We also prove a dichotomy in this section which will be necessary for the final section, as well as providing the promised proof that the definition of $C$-$\ell_p$ upper tree estimates is independent of the choice of weak neighborhood basis $D$. 

The first result of this section is essentially the result of Gale and Stewart \cite{GS} that open games are determined. However, since our proof is short, elucidative, and deals with games on $D$ rather than $\nn$, we include it.

Let $D$ be a directed set.   We define a two player game on $D$. Let us say $B\subset D^{<\nn}$ is \emph{downward closed} if $s,t\in D^{<\nn}$ and $s\leqslant t\in B$, then $s\in B$.  Let $B\subset D^{<\nn}$ be a downward closed subset of $D^{<\nn}$  and let $$\partial B=\{(v_i)_{i=1}^\infty \in D^\nn: (\forall n\in \nn)((v_i)_{i=1}^n\in B)\}.$$    Player I chooses $u_1\in D$ and Player II chooses $v_1\in D$ with $u_1\leqslant v_1$.  Player I chooses $u_2\in D$ and Player II chooses $v_2\in D$ with $u_2\leqslant v_2$. The two players continue in this way until a sequence $(u_i, v_i)_{i=1}^\infty\in (D\times D)^\nn$ is chosen with $u_i\leqslant v_i$ for all $i\in \nn$.  Player I wins if $(v_i)_{i=1}^\infty \in \partial B$, and Player II wins if $(v_i)_{i=1}^\infty \in G:=D^\nn\setminus \partial B$.  We refer to this as the $B$ game.    Note that if $D$ is endowed with the discrete topology and $D^\nn$ is endowed with the product topology, $G$ is open with respect to this topology, and $\partial B$ is closed. 

Let $$S_1=\{\varnothing\}\cup D^{<\nn}$$ and let $$S_2=\{((u_i, v_i)_{i=1}^{n-1}, u): n\in \nn, u_i, v_i, u\in D, u_i\leqslant v_i\}.$$  A \emph{strategy for Player I} is a function $\phi:S_1\to D$. A \emph{strategy for Player II} is a function $\psi:S_2\to D$ such that for every $(t, u)\in S_2$, $u\leqslant \psi(t, u)$.

Fix a strategy $\psi:S_2\to D$ for Player II. We say a sequence $t=(u_i,v_i)_{i=1}^\infty\in (D\times D)^{<\nn}$ is $\psi$-\emph{admissible} provided that for each $j\in \nn$, $v_j=\psi(t|_{j-1}, u_j)$. We say $\psi$ is a \emph{winning strategy for Player II in the} $B$ \emph{game} provided that for any $\psi$-admissible sequence $(u_i, v_i)_{i=1}^\infty$, $(v_i)_{i=1}^\infty\in G$. 

Fix a strategy $\phi:S_1\to D$ for Player I.  We say $(v_i)_{i=1}^n\in D^{<\nn}$ is $\phi$-\emph{admissible} provided that $\phi((v_i)_{i=1}^{j-1})\leqslant v_j$ for every $1\leqslant j\leqslant n$. We say $(v_i)_{i=1}^\infty \in D^\nn$ is $\phi$-admissible if $(v_i)_{i=1}^n$ is $\phi$-admissible for all $n\in \nn$.  We say $\phi$ is a \emph{winning strategy for Player I in the} $B$ \emph{game} provided that any $\phi$-admissible sequence $(v_i)_{i=1}^\infty$ lies in $\partial B$. We say $\phi$ is a \emph{defensive strategy for Player II in the} $B$ \emph{game} provided that for any $n\in \nn$ and any $\phi$-admissible sequence $(v_i)_{i=1}^n$, Player II does not have a winning strategy in the $B_{(v_i)_{i=1}^n}$ game. Here, $$B_{(v_i)_{i=1}^n}=\{s\in D^{<\nn}: (v_i)_{i=1}^n\smallfrown s\in B\}.$$

\begin{lemma}Fix a downward closed subset $B$ of $D^{<\nn}$. \begin{enumerate}[(i)]\item Any defensive strategy for Player I in the $B$ game is a winning strategy for Player I in the $B$ game. \item Either Player II has a winning strategy in the $B$ game or Player I has a defensive strategy in the $B$ game.

\end{enumerate} 

In particular, the $B$ game is determined. 

\label{determined}
\end{lemma}

\begin{proof}$(i)$ Suppose $\phi$ is a defensive strategy for Player I in the $B$ game. Suppose $(v_i)_{i=1}^\infty$ is $\phi$-admissible. Then $(v_i)_{i=1}^n\in B$ for all $n\in \nn$.  Indeed, suppose there exists some $n$ such that $(v_i)_{i=1}^n\in D^{<\nn}\setminus B$. Note that Player II has a strategy in the $B_{(v_i)_{i=1}^n}$ game (indeed, one can take $\psi(t, u)=u$ for any $t$). This strategy is a winning strategy for Player II in the $B_{(v_i)_{i=1}^n}$ game, since $B_{(v_i)_{i=1}^n}=\varnothing$. This contradicts the definition of a defensive strategy, whence we deduce that $(v_i)_{i=1}^n\in B$ for all $n\in \nn$.  Thus $(v_i)_{i=1}^\infty \in \partial B$, and $\phi$ is a winning strategy for Player I in the $B$ game.

$(ii)$ Assume that Player II has no winning strategy in the $B$ game. Then it must be the case that there exists $u\in D$ such that for any $u\leqslant v\in D$, Player II does not have a winning strategy in the $B_{(v)}$ game. Indeed, if for every $u\in D$, there exists $v_u\in D$ with $u\leqslant v_u$ such that Player II has a winning strategy $\psi_u$ in the $B_{(v_u)}$ game, then we can easily construct a winning strategy for Player II in the $B$ game. We define $\psi(\varnothing, u)=v_u$, $\psi((u, v_u)\smallfrown t, u')=\psi_u(t, u')$, and for any $(u,v)\smallfrown t$ such that $v\neq v_u$, define $\psi((u, v)\smallfrown t, u')=u'$. It is easy to see that this is a winning strategy for Player II, contradicting our initial assumption. Therefore there does exist some $u\in D$ with the aforementioned property, and we define $\phi(\varnothing)$ to be this $u\in D$. 

Now assume that for some $(v_i)_{i=1}^n\in D^{<\nn}$ and every $0\leqslant j<n$, $\phi((v_i)_{i=1}^j)$ has been defined. Assume also that if $(v_i)_{i=1}^n$ is such that $\phi((v_i)_{i=1}^{j-1})\leqslant v_j$ for each $1\leqslant j<n$, Player II has no winning strategy in the $B_{(v_i)_{i=1}^n}$ game. If $(v_i)_{i=1}^n$ is such that $\phi((v_i)_{i=1}^{j-1})\leqslant v_j$ for each $1\leqslant j<n$, then we may argue as in the previous paragraph to deduce the existence of some $u\in D$ such that for any $u\leqslant v_{n+1}\in D$, Player II does not have a winning strategy in the $B_{(v_i)_{i=1}^{n+1}}$ game. In this case, we define $\phi((v_i)_{i=1}^n)=u$. If $(v_i)_{i=1}^n$ does not have the property that $\phi((v_i)_{i=1}^{j-1})\leqslant v_j$ for each $1\leqslant j<n$, let $\phi((v_i)_{i=1}^n)$ be any member of $D$. This completes the recursive definition of $\phi$, which is easily seen to be a defensive strategy for Player I.

\end{proof}

\begin{proposition} Let $B\subset D^{<\nn}$ be downward closed. Player I has a winning strategy in the $B$ game if and only if $B$ is big. 

\label{strategy}

\end{proposition}

\begin{proof}$(i)$ Assume $B$ is big and let $I$ be an inevitable subset of $B$. Since $B\cap D^1$ is eventual, there exists $\phi(\varnothing)\in D^1$ such that for any $\phi(\varnothing)\leqslant v$, $(v)\in I$. Now assume that for some $(v_i)_{i=1}^n\in S_1$,  $\phi((v_i)_{i=1}^m)$ has been defined for each $0\leqslant m<n$. Assume also that if $\phi((v_i)_{i=1}^{m-1})\leqslant v_m$ for each $1\leqslant m<n$, then $(v_i)_{i=1}^n\in I$.  In the case that $\phi((v_i)_{i=1}^{m-1})\leqslant v_m$ for each $1\leqslant m<n$, since $(v_i)_{i=1}^n\in I$, there exists $\phi((v_i)_{i=1}^n)\in D$ such that for any $\phi((v_i)_{i=1}^n)\leqslant v_{n+1}\in D$, $(v_i)_{i=1}^{n+1}\in D$. If $(v_i)_{i=1}^n$ does not have the property that $\phi((v_i)_{i=1}^{m-1})\leqslant v_m$ for each $1\leqslant m<n$, let $\phi((v_i)_{i=1}^n)$ be any member of $D$.  This completes the definition of $\phi$, which is easily seen to be a winning strategy for Player I in the $B$ game.

$(ii)$ Assume that $\phi$ is a winning strategy for Player I in the $B$ game. Let $I$ denote the set of all finite, $\phi$-admissible sequences. Then $I\subset D^{<\nn}$ is inevitable and contained in $B$. It is immediate from the definition that $I$ is inevitable. To see that $I$ is contained in $B$, fix $(v_i)_{i=1}^n\in I$. We recursively define $v_{n+1}, v_{n+2}, \ldots$ by $v_{n+k}=\phi((v_i)_{i=1}^{n+k-1})$. Then $(v_i)_{i=1}^\infty$ is $\phi$-admissible, and therefore $(v_i)_{i=1}^n\in B$.

\end{proof}

\begin{corollary} For $1<p<\infty$, $C>0$, and an operator $A:X\to Y$, $A$ fails to satisfy $C$-$\ell_p$ upper tree estimates if and only if there exists a normally weakly null $(x_t)_{t\in D^{<\nn}}\subset B_X$ such that for every $t=(v_i)_{i=1}^\infty \in D^\nn$, there exist $n\in \nn$ and $(a_i)_{i=1}^n\in B_{\ell_p^n}$ such that $$\|A\sum_{i=1}^n a_i x_{t|_i}\|>C.$$  

The analogous result holds for $C$-$c_0$ upper tree estimates. 

\label{dichotomy}

\end{corollary}

\begin{proof} If $B\subset D^{<\nn}$ is big, then there exists $(v_i)_{i=1}^\infty$ such that $(v_i)_{i=1}^n\in B$ for all $n\in \nn$. Assume there exists $(x_t)_{t\in D^{<\nn}}\subset B_X$ normally weakly null such that for every $t=(v_i)_{i=1}^\infty\in D^\nn$, there exist $n\in \nn$ and $(a_i)_{i=1}^n\in B_{\ell_p^n}$ such that $$\|A\sum_{i=1}^n a_i x_{t|_i}\|>C.$$  From this and the comment at the beginning of the proof, it follows that $$B:=\bigcup_{n=1}^\infty \Bigl\{t\in D^n: (\forall (a_i)_{i=1}^n\in \mathbb{K}^n)\Bigl(\|A\sum_{i=1}^n a_i x_{t|_i}\|\leqslant C\|(a_i)_{i=1}^n\|_{\ell_p^n}\Bigr)\Bigr\}$$ cannot be big. Thus $A$ fails to have $C$-$\ell_p$ upper tree estimates.

Now assume that $A$ fails to have $C$-$\ell_p$ upper tree estimates.  This means there exists $(x'_t)_{t\in D^{<\nn}}\subset B_X$ normally weakly null such that $$B:=\bigcup_{n=1}^\infty \Bigl\{t\in D^n: (\forall (a_i)_{i=1}^n\in B_{\ell_p^n})\Bigl(\|A\sum_{i=1}^n a_i x'_{t|_i}\|\leqslant C\Bigr)\Bigr\}$$ is not big.  This means Player II has a winning strategy $\psi$ in the $B$ game.  Define maps $\theta:D^{<\nn}\to D^{<\nn}$ and $\Theta:D^{<\nn}\to (D\times D)^{<\nn}$ by letting $$\theta((u))= (\psi(\varnothing, u)),\hspace{10mm} \Theta((u))=(u, \theta((u))),$$ and if $\Theta(t)$ has been defined, let $$\theta(t\smallfrown (u))= \theta(t)\smallfrown (\psi(\Theta(t), u)),\hspace{10mm} \Theta(t\smallfrown (u))= \Theta(t)\smallfrown (u, \psi(\Theta(t), u)).$$   Then let $x_t= x'_{\theta(t)}$.  It follows from the definitions that $(x_t)_{t\in D^{<\nn}}$ is the collection we seek.

Replacing $\ell_p^n$ with $\ell_\infty^n$ gives the analogous result for $c_0$.

\end{proof}

Let $X$ be a Banach space and let $P$ be a property which finite sequences in $X$ may or may not possess. For example, once an operator $A:X\to Y$ and constants $1<p\leqslant \infty$ and $C>0$ are fixed,  $P$ may be the property that for all $(a_i)_{i=1}^n\in \mathbb{K}^n$, $$\|A\sum_{i=1}^n a_i x_i\|\leqslant C\|(a_i)_{i=1}^n\|_{\ell_p^n}.$$   We will then agree that an infinite sequence $(x_i)_{i=1}^\infty $ has property $P$ if $(x_i)_{i=1}^n$ has property $P$ for every $n\in \nn$.    What we have shown above easily generalizes to the following. 

\begin{corollary} Let $X$ be a Banach space and let $P$ be a property as above. Let $D$ be a fixed weak neighborhood basis at $0$ in $X$. Then exactly one of the two following alternatives hold: 

\begin{enumerate}[(i)]\item For every normally weakly null $(x_t)_{t\in D^{<\nn}}\subset B_X$, the set $$B:=\bigcup_{n=1}^\infty \Bigl\{t\in D^n: (x_{t|_i})_{i=1}^n \text{\ has\ }P\Bigr\}$$ is big.  \item There exists  a normally weakly null collection $(x_t)_{t\in D^{<\nn}}\subset B_X$ such that for every $t\in D^\nn$,  $(x_{t|_i})_{i=1}^\infty$ fails to have $P$. \end{enumerate}

\label{generalize}
\end{corollary}

Now if $(x_t)_{t\in D^{<\nn}}$ is as in $(ii)$ of Corollary \ref{generalize} and if $D_1$ is any other weak neighborhood basis at $0$ in $X$, we may fix $\phi:D_1\to D$ such that $\phi(u)\subset u$ for each $u\in D_1$ and define $\Phi((u_i)_{i=1}^n)=(\phi(u_i))_{i=1}^n$. Then with $x'_t=x_{\Phi(t)}$, $(x'_t)_{t\in D_1^{<\nn}}\subset B_X$ is also as in $(ii)$ with $D$ replaced by $D_1$. Thus the alternatives in Corollary \ref{generalize} are independent of the particular choice of $D$. In particular, we have the following.

\begin{corollary} The definition of $C$-$\ell_p$ (resp. $C$-$c_0$) upper tree estimates is independent of the weak neighborhood basis $D$. 
\label{shortcut}
\end{corollary}

We last conclude 

\begin{corollary} For any operator $A:X\to Y$ and $C>0$, $A:X\to Y$ satisfies $C$-$\ell_p$ (resp. $C$-$c_0$) upper tree estimates if and only if every normally weakly null collection in $B_X$ has a branch $C$-dominated by the $\ell_p$ (resp. $c_0$) basis. 

\end{corollary}

\section{An observation on injective tensor products}

We discuss now some of the known results regarding asymptotically uniformly smooth norms.  We briefly recall the Szlenk index of an operator. Given $\ee>0$, a Banach space $X$, and $K\subset X^*$ weak$^*$-compact, we let $s_\ee(K)$ denote the set of those $x^*\in K$ such that for every weak$^*$-neighborhood $V$ of $x^*$, $\text{diam}(K\cap V)>\ee$.  We define the transfinite derivations by $s^0_\ee(K)=K$, $s^{\xi+1}_\ee(K)=s_\ee(s^\xi_\ee(K))$ for an ordinal $\xi$, and $s^\xi_\ee(K)=\cap_{\zeta<\xi}s^\zeta_\ee(K)$ when $\xi$ is a limit ordinal.   We let $Sz(K, \ee)$ be the minimum ordinal $\xi$  such that $s^\xi_\ee(K)=\varnothing$, provided such a $\xi$ exists, and we write $Sz(K,\ee)=\infty$ if  no such ordinal exists. Note that by weak$^*$-compactness, $Sz(K, \ee)$ cannot be a limit ordinal.

 Given an operator $A:X\to Y$, we let $Sz(A, \ee)=Sz(A^*B_{Y^*}, \ee)$. We let $Sz(A)=\sup_{\ee>0}Sz(A, \ee)$, with the agreement that $Sz(A)=\infty$ provided that there exists $\ee>0$ such that $Sz(A, \ee)=\infty$. Given a Banach space $X$, we let $Sz(X, \ee)=Sz(I_X, \ee)$ and $Sz(X)=Sz(I_X)$. If $Sz(A, \ee)<\omega$ for all $\ee>0$, we let $$\textbf{p}(A)=\underset{\ee\to 0^+}{\lim\sup} \frac{\log Sz(A, \ee)}{|\log(\ee)|}.$$ This quantity need not be finite for operators, although it is finite for Banach spaces.   If $Sz(A, \ee)> \omega$ for some $\ee>0$, we let $\textbf{p}(A)=\infty$.  Given a Banach space $X$, we let $\textbf{p}(X)=\textbf{p}(I_X)$.  We collect known results about these quantities. These results are collected from \cite{Br}, \cite{GKL1}, and \cite{C}. In what follows, $p'$ denotes the conjugate exponent to $p$.

\begin{theorem} Let $A:X\to Y$ be an operator. \begin{enumerate}[(i)]\item $A$ is asymptotically uniformly smoothable if and only if $Sz(A)\leqslant \omega$. \item If $Sz(A)<\omega$, then $A$ is compact and asymptotically uniformly flat. \item If $Sz(A)=\omega$, $A$ is $p$-asymptotically uniformly smoothable for some $1<p<\infty$ if and only if $\textbf{\emph{p}}(A)<\infty$, and in this case $\textbf{\emph{p}}(A)\geqslant 1$ and $$\sup\{p:A\text{\ is\ }p\text{-asymptotically uniformly smoothable}\}= \textbf{\emph{p}}(A)'.$$ \item If $A=I_X$, then $\textbf{\emph{p}}(X)<\infty$ if $Sz(X)\leqslant \omega$. \end{enumerate}

\label{pullman}

\end{theorem}

We also collect some standard facts about duality between $\rho(\cdot, A)$ and $\delta^{\text{weak}^*}(\cdot, A)$.

\begin{proposition} Let $A:X\to Y$ be an operator. \begin{enumerate}[(i)]\item $A$ is asymptotically uniformly smooth if and only if $\delta^{\text{\emph{weak}}^*}(\tau, A)>0$ for all $\tau>0$. \item If $1<p,q<\infty$ and $1/p+1/q=1$, then $A$ is $p$-asymptotically uniformly smooth if and only if $\sup_{\sigma>0} \rho(\sigma, A)/\sigma^p<\infty$ if and only if there exists $\sigma_1>0$ such that $\sup_{0<\sigma<\sigma_1} \rho(\sigma, A)/\sigma^p<\infty$ if and only if $\inf_{0<\tau<1}\delta^{\text{\emph{weak}}^*}(\tau, A)/\tau^q>0$ if and only if for every $\tau_1>0$, $\delta^{\text{\emph{weak}}^*}(\tau, A)/\tau^q<\infty$. \item $A$ is asymptotically uniformly flat if and only if $\inf_{\tau>0} \delta^{\text{\emph{weak}}^*}(\tau, A)/\tau>0$. \end{enumerate}

\label{duality}

\end{proposition}

Items $(i)$ and $(ii)$ were shown in \cite{CD}.  Since asymptotic uniform flatness for operators has not been previously defined, item $(iii)$ has not appeared in the literature, although the spatial analogue is well-known. We provide the short proof, which requires the following.

\begin{proposition}\cite{CD} Fix $\sigma, \tau>0$. If $A:X\to Y$ is an operator and $\delta^{\text{\emph{weak}}^*}(\tau, A)\geqslant \sigma\tau$, then for any $c>0$, any net $(y^*_\lambda)\subset B_{Y^*}$ converging weak$^*$ to $y^*$ such that $\underset{\lambda}{\lim\inf} \|A^*y^*_\lambda-A^*y^*\|= c\geqslant \tau$, then $\sigma c\leqslant 1$ and $\|y^*\|\leqslant 1-\sigma c$. 

\label{wasc}

\end{proposition}

\begin{proof} In \cite[Proposition $3.10$]{CD}, it was shown under these hypotheses that either $y^*=0$ or $\|y^*\|\leqslant \sigma c$.  Thus it remains to show that if $y^*=0$, then $\sigma c\leqslant 1$.   Under the hypotheses, it is clear that $A$ is not compact, whence $\delta^{\text{weak}^*}(\cdot, A)$ takes on finite values and is continuous. Fix any $0<\sigma_1<\sigma$ and fix $0<\tau_1<\tau$ such that $\delta^{\text{weak}^*}(\tau_1, A)\geqslant \sigma_1\tau_1$. Fix $0<\delta<1$ such that $(1+\delta)\tau_1<\tau$ and $y^*_0\in \delta S_{Y^*}$. Now let $z^*=(1+\delta)^{-1}y^*_0$, $z^*_\lambda= (1+\delta)^{-1} (y^*_0+y^*_\lambda)$.  Applying the result to $(z^*_\lambda)\subset B_{Y^*}$, $z^*$ with $\sigma, \tau$ replaced by $\sigma_1$, $\tau_1$, we deduce that $\sigma_1\bigl(\frac{c}{1+\delta})\leqslant 1$. Since $0<\sigma_1<\sigma$ and $0<\delta<1$ were arbitrary, we deduce that $\sigma c\leqslant 1$.

\end{proof}

\begin{proof}[Proof of Proposition \ref{duality}$(iii)$] If $A$ is the zero operator, there is nothing to show, so assume $A$ is not the zero operator.

Assume $\rho(\sigma, A)=0$. Fix $\tau>0$, $y^*\in S_{Y^*}$, and $(y^*_\lambda)\subset Y^*$ is weak$^*$-null and $\|A^*y^*_\lambda\|\geqslant \tau$.  Fix $\delta>0$ and $y\in S_Y$ such that $\text{Re\ }y^*(y)>1-\delta$. After passing to a subnet, we may fix $(x_\lambda)\subset \sigma B_X$ such that $\underset{\lambda}{\lim\inf}\text{Re\ }y^*_\lambda(Ax_\lambda)\geqslant \sigma\tau/2$, whence $$\underset{\lambda}{\lim\inf} \|y^*+y^*_\lambda\| \geqslant \underset{\lambda}{\lim\inf} \text{Re\ }(y^*+y^*_\lambda)(y+Ax_\lambda)\geqslant 1-\delta+\sigma\tau/2.$$  This shows that $\delta^{\text{weak}^*}(\tau, A)/\tau\geqslant \sigma/2$ for all $\tau>0$.

Now suppose $\sigma:=\inf_{\tau>0}\delta^{\text{weak}^*}(\tau, A)/\tau>0$.  Assume $\rho(\sigma, A)>0$ and fix $y\in S_Y$, $(x_\lambda)\subset \sigma B_X$ such that $$\mu:=\inf_\lambda \|y+Ax_\lambda\|>1.$$   For each $\lambda$, fix $y^*_\lambda\in B_{Y^*}$ such that $\text{Re\ }y^*_\lambda(y+Ax_\lambda)\geqslant \mu$. By passing to a subnet, we may assume $\text{Re\ }y^*_\lambda(Ax_\lambda/\sigma)\to \tau$ and $y^*_\lambda\underset{\text{weak}^*}{\to}y^*$.  Note that $$\underset{\lambda}{\lim\inf} \|A^*y^*_\lambda - A^*y^*\| \geqslant \underset{\lambda}{\lim\inf} \text{Re\ }(A^*y^*_\lambda-A^*y^*)(x_\lambda/\sigma)= \tau.$$ Since $$1<\mu \leqslant \|y^*\|\|y\|+\sigma\tau,$$ $\tau>0$.  By Proposition \ref{wasc}, $\|y^*\|\leqslant 1-\sigma \tau$ and $$1+\mu \leqslant \underset{\lambda}{\lim\inf} \text{Re\ }y^*_\lambda(y+Ax_\lambda) = \text{Re\ }y^*(y)+\sigma \tau \leqslant \|y^*\|\|y\|+\sigma\tau \leqslant 1-\sigma\tau + \sigma\tau=1,$$ a contradiction.

\end{proof}

We recall that for Banach spaces, $X_0, X_1$, $X_0\hat{\otimes}_\ee X_1$ is the norm closure in $\mathfrak{L}(X^*_1, X_0)$ of the operators $u$ from $X^*_1$ into $X_0$ of the form $u=\sum_{i=1}^n x_{0, i}\otimes x_{1, i}$, where $(\sum_{i=1}^n x_{0,i}\otimes x_{1,i})(x^*_1)=\sum_{i=1}^n x^*_1(x^*_{1,i})x_{0,i}$.  If $A_0:X_0\to Y_0$ and $A_1:X_1\to Y_1$ are two operators, there is an induced an operator $A_0\otimes A_1:X_1\hat{\otimes}_\ee X_1\to Y_0\hat{\otimes }_\ee Y_1$ which is the linear extension of the map $x_0\otimes x_1\mapsto A_0 x_0\otimes A_1 x_1$. If $A_i=I_{X_i}$, then $A_0\otimes A_1= I_{X_0\hat{\otimes}_\ee X_1}$. We wish to show that each of our asymptotic smoothness properties pass from $A_0, A_1$ to $A_0\otimes A_1$.

\begin{lemma} For $i=0,1$, let $A_i:X_i\to Y_i$ be an operator. Let $R=\max\{1,\|A_0\|, \|A_1\|\}$.  For $\tau>0$, let $$\delta(\tau)= \min \{\delta^{\text{\emph{weak}}^*}(\tau, A_0), \delta^{\text{\emph{weak}}^*}(\tau, A_1)\}.$$  

For $0<\sigma, \tau$, if $\delta(\tau)\geqslant \sigma \tau$. Then $\rho(\sigma/2R, A_0\otimes A_1)\leqslant \sigma\tau$. 

\label{daniel}

\end{lemma}

\begin{proof} Suppose not. Assume $\delta(\tau) \geqslant \sigma \tau$ and $\rho(\sigma, A_0\otimes A_1)> \sigma \tau$. Then there exist $\mu>1+\sigma\tau$, $u\in B_{Y_0\hat{\otimes}_\ee Y_1}$, and a weakly null net $(u_\lambda)\subset \frac{\sigma}{2R} B_{X_0\hat{\otimes}_\ee X_1}$ such that $\mu<\|u+A_0\otimes A_1 u_\lambda\|$ for all $\lambda$.    For each $\lambda$, fix $y^*_{0, \lambda}\in B_{Y^*_0}$, $y^*_{1, \lambda}\in B_{Y^*_1}$ such that $$\text{Re\ }y^*_{0, \lambda}\otimes y^*_{1, \lambda}(u+A_0\otimes A_1 u_\lambda)=\|u+A_0\otimes A_1 u_\lambda\|.$$  By passing to a subnet, we may assume there exist $y^*_0\in B_{Y^*_0}$, $y^*_1\in B_{Y^*_1}$, and $b$ such that $$y^*_{0, \lambda}\underset{\text{weak}^*}{\to} y^*_0,$$ $$y^*_{1, \lambda}\underset{\text{weak}^*}{\to} y^*_1,$$  $$A^*_0y^*_{0, \lambda}\otimes A^*_1 y^*_{1, \lambda}(2Ru_\lambda/\sigma)=y^*_{0, \lambda}\otimes  y^*_{1, \lambda}(A_0\otimes A_1 (2Ru_\lambda/\sigma))\to b.$$

 Note that $$1+\sigma \tau<\mu \leqslant \lim_\lambda \text{Re\ }y^*_{0, \lambda}\otimes y^*_{1, \lambda}(u+A_0\otimes A_1 u_\lambda) = \text{Re\ }y^*_0\otimes y^*_1(u)+\sigma b/2R\leqslant 1+\sigma b/2R.$$ From this we deduce that $b>0$. Note also that \begin{align*} \underset{\lambda}{\lim\sup} & \|A^*_0y^*_{0, \lambda}\otimes A^*_1 y^*_{1, \lambda}-A^*_0y^*_0\otimes A^*_1 y^*_1\| \\ & \geqslant \lim_\lambda \text{Re\ }(A^*_0y^*_{0, \lambda}\otimes A^*_1 y^*_{1, \lambda}-A^*_0y^*_0\otimes A^*_1 y^*_1)(2Ru_\lambda/\sigma)=b.\end{align*}  Now note that \begin{align*} b  \leqslant & \underset{\lambda}{\lim\sup}  \|A^*_0y^*_{0, \lambda}\otimes A^*_1 y^*_{1, \lambda}-A^*_0y^*_0\otimes A^*_1 y^*_1\| \\ & \leqslant \underset{\lambda}{\lim\sup} \|A^*_0y^*_{0, \lambda}\otimes A^*_1 y^*_{1, \lambda}-A^*_0y^*_0\otimes A^*_1 y^*_{1, \lambda}\| \\ & + \underset{\lambda}{\lim\sup} \|A^*_0y^*_0\otimes A^*_1 y^*_{1, \lambda}-A^*_0y^*_0\otimes A^*_1 y^*_1\| \\ & = \underset{\lambda}{\lim\sup} \|A^*_0y^*_{0, \lambda}-A^*_0y^*_0\|\|A^*_1y^*_{1, \lambda}\| \\ & + \underset{\lambda}{\lim\sup} \|A^*_0 y^*_0\|\|A^*_1 y^*_{1, \lambda}-A^*_1 y^*_1\| \\ & \leqslant 2R\max\{\underset{\lambda}{\lim\sup} \|A^*_0y^*_{0, \lambda}-A^*_0y^*_0\|, \underset{\lambda}{\lim\sup} \|A^*_1 y^*_{1, \lambda}-A^*_1y^*_1\|\}. \end{align*}  By passing to a subnet once more and switching $A_0$ and $A_1$ if necessary, we may assume $$\frac{b}{2R} \leqslant \underset{\lambda}{\lim\inf}\|A^*_0 y^*_{0, \lambda}-A^*_0 y^*_0\|.$$    If $b/2R\geqslant \tau$, $\|y^*_0\otimes y^*_1\|\leqslant \|y^*_0\|\leqslant 1-\sigma b/2R$  by Proposition \ref{wasc}.  From this we deduce that $$1+\sigma\tau< \|y^*_0\otimes y^*_1\|\|u\|+\sigma b/2R  \leqslant 1-\sigma b/2R+\sigma b/2R=1$$ a contradiction.  If $b/2R<\tau$, then $$1+\sigma \tau \leqslant \|y^*_0\otimes y^*_1\|\|u\|+ \sigma b/2R<1+\sigma \tau, $$ another contradiction.

\end{proof}

We say a property which an operator may or may not possess  \emph{passes to injective tensor products of operators} if for any two operators $A_0:X_0\to Y_0$,$ A_1:X_1\to Y_1$ with this property, $A_0\otimes A_1:X_0\hat{\otimes}_\ee X_1\to Y_0\hat{\otimes}_\ee Y_1$ has this property.  We say a property which Banach spaces may or may not possess \emph{passes to injective tensor products of Banach spaces} if for any two Banach spaces $X_0$, $X_1$ with this property, $X_0\hat{\otimes}_\ee X_1$ has this property. 

\begin{corollary} For any $1<p<\infty$, the following properties pass to injective tensor products of operators and of Banach spaces. \begin{enumerate}[(i)]\item Asymptotic uniform smoothness. \item $p$-asymptotic uniform smoothness. \item Asymptotic uniform flatness. \item Asymptotic uniform smoothability. \item $p$-asymptotic uniform smoothability. \item Asymptotic uniform flatness. \item $\textbf{\emph{p}}(\cdot)\leqslant p$. 
\end{enumerate}

\label{tensor}

\end{corollary}

\begin{proof} We prove $(i)$-$(iii)$. Items $(iv)$-$(vi)$ follows from $(i)$-$(iii)$ and renorming, while item $(vii)$ follows from item $(v)$ and Theorem \ref{pullman}.

In the proof, suppose $A_0$, $A_1$, $R$, and $\delta$ are as in Lemma \ref{daniel}. We go through the usual Young duality. However, since we must go through the function $\delta$, we include the proof. We will use Proposition \ref{duality} throughout. 

If $A_0, A_1$ are asymptotically uniformly smooth, $\delta(\tau)>0$ for all $\tau>0$. Fix $\tau>0$ and let $\sigma=\delta(\tau)/2R\tau>0$.  Then $\delta(\tau)= (2R\sigma)\tau$, whence $\rho(\sigma, A_0\otimes A_1)\leqslant 2R\sigma \tau$.  Since $\tau>0$ was arbitrary, $\inf_{\sigma>0} \rho(\sigma, A_0\otimes A_1)= 0$. But $\rho(\cdot, A_0\otimes A_1)$ is convex and $\rho(0, A_0\otimes A_1)=0$, whence $$\lim_{\sigma\to 0^+}\rho(\sigma, A_0\otimes A_1)/\sigma=\inf_{\sigma>0} \rho(\sigma, A_0\otimes A_1)/\sigma=0.$$  Thus $A_0\otimes A_1$ is asymptotically uniformly smooth.

Now assume that for $1<p,q<\infty$ with $1/p+1/q=1$, $A_0$, $A_1$ are $p$-asymptotically uniformly smooth. Then there exists $0<c<1$ such that for all $0<\tau<1$, $\delta(\tau)\geqslant c\tau^q$.  Fix $0<\sigma<c/2R$ and let $c'=2R(2R/c)^{\frac{1}{q-1}}$ and $\tau=(2R\sigma/c)^\frac{1}{q-1}\in (0,1)$.  Then $$\delta(\tau)\geqslant c \tau^q = 2R\sigma \tau.$$ From this we deduce that $$\rho(\sigma, A_0\otimes A_1) \leqslant 2R\sigma \tau = c' \sigma^p .$$ Thus we deduce that $\rho(\sigma, A_0\otimes A_1)\leqslant c'\sigma^p$ for all $0<\sigma <c/2R$. Thus $\sup_{\sigma>0} \rho(\sigma, A_0\otimes A_1)/\sigma^p<\infty$, and $A_0\otimes A_1$ is $p$-asymptotically uniformly smooth. 

Now assume that $A_0$, $A_1$ are asymptotically uniformly flat. Then there exists $0<c<1$ such that $\delta(\tau)\geqslant c\tau$ for all $\tau>0$. Then with $\sigma=c/2R$, since $\delta(\tau)\geqslant 2R\sigma \tau$ for all $\tau>0$, $\rho(\sigma, A_0\otimes A_1)\leqslant 2R\sigma \tau$ for all $\tau>0$. Since this holds for all $\tau>0$, $\rho(\sigma, A_0\otimes A_1)=0$. Thus $A_0\otimes A_1$ is asymptotically uniformly flat.

\end{proof}

\begin{rem}\upshape Of course, there is a partial converse to Corollary \ref{tensor}: If $A_0\otimes A_1$ possesses any of the seven properties listed in Corollary \ref{tensor}, then either $A_0=0$, $A_1=0$, or both $A_0$ and $A_1$ possess that same property.

\end{rem}

\section{Ideals}

In this section, we let $\textbf{Ban}$ denote the class of all Banach spaces over $\mathbb{K}$. We let $\mathfrak{L}$ denote the class of all operators between Banach spaces and for $X,Y\in \textbf{Ban}$, we let $\mathfrak{L}(X,Y)$ denote the set of operators from $X$ into $Y$. For $\mathfrak{I}\subset \mathfrak{L}$ and $X,Y\in \textbf{Ban}$, we let $\mathfrak{I}(X,Y)=\mathfrak{I}\cap \mathfrak{L}(X,Y)$. We recall that a class $\mathfrak{I}$ is called an ideal if \begin{enumerate}[(i)]\item For any $W,X,Y,Z\in \textbf{Ban}$, any $C\in \mathfrak{L}(W,X)$, $B\in \mathfrak{I}(X,Y)$, and $A\in \mathfrak{L}(Y,Z)$, $ABC\in \mathfrak{I}$, \item $I_\mathbb{K}\in \mathfrak{I}$, \item for each $X,Y\in \textbf{Ban}$, $\mathfrak{I}(X,Y)$ is a vector subspace of $\mathfrak{L}(X,Y)$. \end{enumerate}

We recall that an ideal $\mathfrak{I}$ is said to be  \emph{closed} provided that for any $X,Y\in \textbf{Ban}$, $\mathfrak{I}(X,Y)$ is closed in $\mathfrak{L}(X,Y)$ with its norm topology. 

If $\mathfrak{I}$ is an ideal and $\iota$ assigns to each member of $\mathfrak{I}$ a non-negative real value, then we say $\iota$ is an \emph{ideal norm} provided that \begin{enumerate}[(i)]\item for each $X,Y\in \textbf{Ban}$, $\iota$ is a norm on $\mathfrak{I}(X,Y)$, \item for any $W,X,Y,Z\in \textbf{Ban}$ and any $C\in \mathfrak{L}(W,X)$, $B\in \mathfrak{I}(X,Y)$, $A\in \mathfrak{I}(Y,Z)$, $\iota(ABC)\leqslant \|A\|\iota(B)\|C\|$, \item for any $X,Y\in \textbf{Ban}$, any $x\in X$, and any $y\in Y$, $\iota(x\otimes y)=\|x\|\|y\|$. \end{enumerate}

If $\mathfrak{I}$ is an ideal and $\iota$ is an ideal norm on $\mathfrak{I}$, we say $(\mathfrak{I}, \iota)$ is a \emph{Banach ideal} provided that for every $X,Y\in \textbf{Ban}$, $(\mathfrak{I}(X,Y), \iota)$ is a Banach space.

For $1<p\leqslant\infty$, we let $T_p(A)$ denote the infimum of those $C$ such that $A$ satisfies $C$-$\ell_p$ upper tree estimates (resp. $C$-$c_0$ upper tree estimates if $p=\infty$). We observe the convention that $T_p(A)$ is defined for all operators, and $T_p(A)=\infty$ if there exists no $C$ such that $A$ satisfies $C$-$\ell_p$ upper tree estimates.   Let $\mathfrak{t}_p(A)=\|A\|+T_p(A)$. Let $\mathfrak{T}_p$ denote the class of operators $A$ such that $\mathfrak{t}_p(A)<\infty$.  By Theorem \ref{renorming1}, for $1<p<\infty$, $\mathfrak{T}_p$ is the class of $p$-asymptotically uniformly smoothable operators and $\mathfrak{T}_\infty$ is the class of asymptotically uniformly flattenable operators.

We will need the following obvious fact. 

\begin{proposition} Let $I$ be a non-empty set and fix $1<p<\infty$. Suppose that for each $i\in I$, $A_i:X_i\to Y_i$ is an operator. Suppose also that $\sup_{i\in I}\|A_i\|<\infty$. Then if $A:X:=(\oplus_{i\in I}X_i)_{\ell_p(I)}\to Y:=(\oplus_{i\in I} Y_i)_{\ell_p(I)}$ is defined by $A|_{X_i}=A_i$, then $$\textbf{\emph{t}}_p(A)\leqslant \sup_{i\in I} \max\{\textbf{\emph{t}}_p(A_i), \|A_i\|\}.$$  

The analogous result holds for $\textbf{\emph{t}}_\infty$ if we replace the $\ell_p(I)$ direct sum with the $c_0(I)$ direct sum.

\end{proposition}

\begin{proof} If the supremum on the right is infinite, there is nothing to show. Assume the supremum is finite. Fix $y=(y_i)_{i\in I}\in Y$ such that $J:=\{i\in I: y_i\neq 0\}$ is finite. Note that the set of such $y$ is dense in $Y$. Now fix $\sigma\geqslant 0$ and a weakly null net $(x_\lambda)\subset \sigma B_X$ and assume $\|y+Ax_\lambda\|\geqslant \mu$ for all $\lambda$. By passing to a subnet, we may assume there exist scalars $a, a_i$, $i\in J$ such that $a=\lim_\lambda \|P_{I\setminus J} x_\lambda\|$ and $a_i=\lim_\lambda \|P_i x_\lambda\|$ for each $i\in J$. Note that $a^p+\sum_{i\in J}a_i^p \leqslant \sigma^p$. Let $C=\sup_{i\in I} \max\{\textbf{t}_p(A_i), \|A_i\|\}$ and note that \begin{align*} \underset{\lambda}{\lim\sup} \|y+Ax_\lambda\|^p & \leqslant \underset{\lambda}{\lim} \|P_{I\setminus J} Ax_\lambda\|^p + \sum_{i\in J}\underset{\lambda}{\lim\sup}\|P_i(y+Ax_\lambda)\|^p \\ & \leqslant \bigl(\sup_{i\in I}\|A_i\|^p\bigr)\sigma^p + \sum_{i\in I}\|y_i\|^p+\textbf{t}_p(A_i)^pa_i^p \\ & \leqslant \|y\|^p+C^p\sigma^p. \end{align*}

\end{proof}

We recall that by the proof of Theorem \ref{renorming1}, $T_p(A)\leqslant \textbf{t}_p(A)$. 

\begin{theorem} For each $1<p\leqslant \infty$, $(\mathfrak{T}_p, \mathfrak{t}_p)$ is a Banach ideal.

\end{theorem}

\begin{proof} It is evident that $T_p(cA)=|c|T_p(A)$ for any scalar.  It is also clear that $T_p$ satisfies the triangle inequality, since the intersection of two big trees is big. From this it follows that $\mathfrak{t}_p$ is a norm on $\mathfrak{T}_p(X,Y)$ for any $X,Y\in \textbf{Ban}$. Moreover,  $T_p(A)=0$ for any compact operator, so that $\mathfrak{T}_p$ contains all finite rank operators and $$\mathfrak{t}_p(x\otimes y)= \|x\otimes y\|+0=\|x\|\|y\|$$ for any $X,Y\in \textbf{Ban}$ and any $x\in X$, $y\in Y$. 

Now suppose that $X,Y,Z\in \textbf{Ban}$ and $A:X\to Y$, $B:Y\to Z$ are operators with $\|B\|=1$.  Fix a normally weakly null $(x_t)_{t\in D^{<\nn}}\subset B_X$, where $D$ is a weak neighborhood basis at $0$ in $X$.  Fix $C>T_p(A)$ and note that \begin{align*} \bigcup_{n=1}^\infty &\Bigl\{t\in D^n: (\forall (a_i)_{i=1}^n\in \mathbb{K}^n)\Bigl(\|BA\sum_{i=1}^n a_i x_{t|_i}\|\leqslant C\Bigr)\Bigr\} \\ & \supset \bigcup_{n=1}^n \Bigl\{t\in D^n: (\forall (a_i)_{i=1}^n\in \mathbb{K}^n)\Bigl(\|A\sum_{i=1}^n a_i x_{t|_i}\|\leqslant C\Bigr)\Bigr\},\end{align*} and by definition the latter set is big. This shows that $T_p(BA)\leqslant T_p(A)$ when $\|B\|=1$, whence by homogeneity we deduce that for any $A:X\to Y$ and $B:Y\to Z$, $T_p(BA)\leqslant \|B\|T_p(A)$.

Now suppose that $W,X,Y\in \textbf{Ban}$ and $B:W\to X$, $A:X\to Y$ operators with $\|B\|=1$.  Given a finite subset $F$ of $W^*$ and $\ee>0$, let $$\mathcal{U}_{F,\ee}=\{w\in W: (\forall w^*\in F)(|w^*(w)|<\ee)\}.$$ Given a finite subset $F$ of $X^*$ and $\ee>0$, let $$\mathcal{V}_{F, \ee}=\{x\in X: (\forall x^*\in F)(|x^*(x)|<\ee)\}.$$ Let $$D_W=\{\mathcal{U}_{F, \ee}: F\subset W^*\text{\ finite},\ee>0\}$$ and $$D_X= \{\mathcal{V}_{F, \ee}: F\subset X^*\text{\ finite}, \ee>0\}.$$  Now assume $T_p(AB)>T_p(A)$. Then by Corollary \ref{dichotomy}, there exist $T_p(AB)>C>T_p(A)$ and a normally weakly null $(w_t)_{t\in D_W^{<\nn}}\subset B_W$ such that for every $(v_i)_{i=1}^\infty\in D^\nn_W$, there exist $n\in \nn$ and $(a_i)_{i=1}^n\in S_{\ell_p^n}$ such that $$\|AB\sum_{i=1}^n a_i w_{t|_i}\|> C\|(a_i)_{i=1}^n\|_{\ell_p^n}.$$  Define $\phi:D_X\to D_W$ by $\phi(\mathcal{V}_{F, \ee})=\mathcal{U}_{A^*F, \ee}$ and define $\Phi:D_X^{<\nn}\to D_W^{<\nn}$ by $\Phi((v_i)_{i=1}^n)=(\phi(u_i))_{i=1}^n$.  Let $x_t= Bw_{\Phi(t)}$. Then $(x_t)_{t\in D_X^{<\nn}}\subset B_X$ is normally weakly null.  Moreover, $$S:=\bigcup_{n=1}^\infty \Bigl\{t\in D^n_X:(\forall (a_i)_{i=1}^n\in \mathbb{K}^n)\Bigl(\|A\sum_{i=1}^n a_i x_{t|_i}\|\leqslant C\Bigr)\Bigr\}$$ cannot be big, otherwise there would exist $(v_i)_{i=1}^\infty\in D_X^\nn$ such that $(v_i)_{i=1}^n\in S$ for all $n\in \nn$. However, by our choice of $(w_t)_{t\in D^{<\nn}_W}$, no such $(v_i)_{i=1}^\infty$ could exist, since then for all $n\in \nn$ and all $(a_i)_{i=1}^n\in \mathbb{K}^n$, $$\|AB\sum_{i=1}^n a_i w_{\Phi(t|_i)}\| = \|A\sum_{i=1}^n a_i x_{t|_i}\| \leqslant C\|(a_i)_{i=1}^n\|_{\ell_p^n},$$ while $(\Phi(t|_i))_{i=1}^\infty \in D^\nn_W$ does not have this property.

The above arguments show that $\mathfrak{t}_p$ is an ideal norm. It remains to show that for any $X,Y$, $(\mathfrak{T}_p(X,Y), \mathfrak{t}_p)$ is a Banach space.  To that end, fix a $\mathfrak{t}_p$-Cauchy sequence $(A_n)_{n=1}^\infty$ in $\mathfrak{T}_p(X,Y)$.  Then this sequence is norm Cauchy, and has a norm limit, say $A$. To obtain a contradiction, suppose there exists $0<\ee<1$ such that $\underset{n}{\lim\sup} \mathfrak{t}_p(A_n-A)> \ee$. Of course, this means that $\underset{n}{\lim\sup} T_p(A_n-A)>\ee$.   Recall the convention that $T_p(B)=\infty$ when $B$ does not lie in $\mathfrak{T}_p$, so that the value $T_p(A_n-A)$ is defined and does not assume that $A\in \mathfrak{T}_p$.   Now fix numbers $(\ee_n)_{n=1}^\infty \subset (0,1)$ such that $4C_p\sum_{n=1}^\infty \ee_n <\ee$.  Here, $C_p\geqslant 1$ is the constant from Remark \ref{useful remark} By passing to a subsequence, we may assume that for every $n\in \nn$, $T_p(A_n-A)>\ee$ and $\mathfrak{t}_p(A_n-A_{n+1})<\ee_n^2$.   Since $\mathfrak{t}_p(\ee_n^{-2}(A_n-A_{n+1}))<1$, by Remark \ref{useful remark}, there exists a norm $|\cdot|_n$ on $Y$ such that $A_n-A_{n+1}:X\to (Y, |\cdot|_n)$ satisfies $$\textbf{t}_p(A_n-A_{n+1}:X\to (Y, |\cdot|_n)) \leqslant C_p$$  and $\frac{1}{2}B_Y\subset B_Y^{|\cdot|_n}\subset 2B_Y$.

Now for each $n\in \nn$, let $X_n=X$ and let $Y_n=(Y, |\cdot|_n)$.  Define the operators $S_1:X\to (\oplus_{n=1}^\infty X_n)_{\ell_p}$, $S_2:(\oplus_{n=1}^\infty X_n)_{\ell_p}\to (\oplus_{n=1}^\infty Y_n)_{\ell_p}$, and $S_3:(\oplus_{n=1}^\infty Y_n)_{\ell_p}\to Y$ by $$S_1x=(\ee_n x)_{n=1}^\infty, \hspace{10mm} S_2((x_n)_{n=1}^\infty)= (\ee_n^{-2}(A_n-A_{n+1})x_n)_{n=1}^\infty,\hspace{10mm} S_3((y_n)_{n=1}^\infty)=\sum_{n=1}^\infty \ee_n y_n.$$  If $p=\infty$, we replace the $\ell_p$ direct sum with the $c_0$ direct sum.  Then $\|S_1\|\leqslant \sum_{n=1}^\infty \ee_n$, $\|S_3\|\leqslant 2\sum_{n=1}^\infty \ee_n$, and \begin{align*} T_p(S_2) & \leqslant \sup_{n\in \nn}\max\{T_p(A_n-A_{n+1}:X_n\to Y_n), \|A_n-A_{n+1}:X_n\to Y_n\|\} \\ &  \leqslant \max \{C_p, 2\}<2C_p.\end{align*}  Moreover, $A_1-A=S_3S_2S_1$ and $$T_p(A_1-A)\leqslant \|S_3\|T_p(A)\|S_1\| \leqslant 4C_p\bigl(\sum_{n=1}^\infty \ee_n\bigr)^2<\ee.$$ This contradiction finishes the proof.

\end{proof}

\begin{rem}\upshape For each $1<p\leqslant \infty$, the class $\mathfrak{T}_p$ fails to be a closed ideal, necessitating the introduction of the norm $\mathfrak{t}_p$.  Indeed, let $\theta_n=1/\log_2(n+1)$ for each $n\in \nn$. Define  $A:(\oplus_{n=1}^\infty C[0, \omega^n])_{c_0}\to (\oplus_{n=1}^\infty C[0, \omega^n])_{c_0}$  by $A|_{C[0, \omega^n]}=\theta_n I_{C[0, \omega^n]}$ and $A_k:(\oplus_{n=1}^\infty C[0, \omega^n])_{c_0}\to (\oplus_{n=1}^\infty C[0, \omega^n])_{c_0}$ is given by $A_k|_{C[0, \omega^n]}=\theta_n I_{C[0, \omega^n]}$ for those $n$ such that $n\leqslant k$ and $A_k|_{C[0, \omega^n]}\equiv 0$ if $n>k$. Then it is easy to verify that $A_n\to A$ in norm, each $A_n$ is asymptotically uniformly flattenable, while $A$ fails to have any non-trivial Szlenk power type. From this it follows that $A$ cannot satisfy $C$-$\ell_p$ upper tree estimates for any $1<p$ and $C>0$.

\end{rem}

\end{document}